\documentclass[11pt]{amsart}
\usepackage{amsmath,amsxtra,amssymb,amsthm,amsfonts, eufrak}
\usepackage{graphicx}
\usepackage{color}
\usepackage{enumitem}
\usepackage[english]{babel}
\usepackage[autostyle]{csquotes}
\usepackage{mathrsfs}
\usepackage{tikz}
\usetikzlibrary{matrix}
\usepackage{comment}
\usepackage{dsfont}

\numberwithin{equation}{section}
\newtheorem{lemma}{Lemma}[section]

\newtheorem{theorem}[lemma]{Theorem}

\newtheorem{cor}[lemma]{Corollary}

\newtheorem{rem}[lemma]{Remark}

\newtheorem{fact}[lemma]{Fact}

\newtheorem{defi}[lemma]{Definition}
\newtheorem{exam}[lemma]{Example}

\renewcommand{\for}{\begin{eqnarray*}}

\newcommand{\mel}{\end{eqnarray*}}

\newcommand{\lel}{\pl = \pl}

\newcommand{\vertiii}[1]{{\left\vert\kern-0.25ex\left\vert\kern-0.25ex\left\vert #1 
		\right\vert\kern-0.25ex\right\vert\kern-0.25ex\right\vert}}

\newcommand{\pl}{\hspace{.1cm}}

\newcommand{\qd}{\end{proof}\vspace{0.5ex}}

\newcommand{\si}{\sigma}

\newcommand{\pf}{\begin{proof}}

\newcommand{\xspace}{\hbox{\kern-2.5pt}}
\newcommand{\xyspace}{\hbox{\kern-1.1pt}}

\definecolor{LightGray}{rgb}{0.94,0.94,0.94}
\definecolor{VeryLightBlue}{rgb}{0.9,0.9,1}
\definecolor{LightBlue}{rgb}{0.8,0.8,1}
\definecolor{DarkBlue}{rgb}{0,0,0.6}
\definecolor{LightGreen}{rgb}{0.88,1,0.88}
\definecolor{MidGreen}{rgb}{0.6,1,0.6}
\definecolor{DarkGreen}{rgb}{0,0.6,0}
\definecolor{DarkGrreen}{rgb}{0,0.8,0}
\definecolor{VeryLightYellow}{rgb}{1,1,0.9}
\definecolor{LightYellow}{rgb}{1,1,0.6}
\definecolor{MidYellow}{rgb}{1,1,0.5}
\definecolor{DarkYellow}{rgb}{0.8,1,0.3}
\definecolor{VeryLightRed}{rgb}{1,0.9,0.9}
\definecolor{LightRed}{rgb}{1,0.8,0.8}
\definecolor{DarkRed}{rgb}{0.8,0.2,0}
\definecolor{DarkRedb}{rgb}{0.6,0.2,0}
\definecolor{DarkLila}{rgb}{0.8,0,1}
\definecolor{Beige}{rgb}{0.96,0.96,0.86}
\definecolor{Gold}{rgb}{1.,0.84,0.}
\definecolor{Goldb}{rgb}{0.7,0.3,0.5}
\definecolor{MyYellow}{rgb}{1.,0.84,0.8}

\begin{document}
\title{On the symmetrized arithmetic-geometric mean inequality for operators}

\author{WAFAA ALBAR}

\author{MARIUS JUNGE$^+$}

\author{MINGYU ZHAO}

\address{KING ABDULAZIZ UNIVERSITY}
\email{walbar@kau.edu.sa}

\address{UNVERSITY OF ILLINOIS AT URBANA-CHAMPAIGN}
\email{junge@math.uiuc.edu}
\thanks{$^+$Partially supported by DMS 1501103 and BigData 1447879}

\address{UNVERSITY OF ILLINOIS AT URBANA-CHAMPAIGN}
\email{mzhao16@illinois.edu}
  
\maketitle

		\mbox{} \hspace{4mm} We study the symmetrized noncommutative arithmetic geometric mean inequality introduced(AGM) by Recht and R\'{e} in \cite{recht2012beneath} $$ 
		\|\frac{(n-d)!}{n!}\sum\limits_{{ j_1,...,j_d \mbox{ \scriptsize different}} }A_{j_{1}}^*A_{j_{2}}^*...A_{j_{d}}^*A_{j_{d}}...A_{j_{2}}A_{j_{1}} \| \leq C(d,n) \|\frac{1}{n} \sum_{j=1}^n A_j^*A_j\|^d .$$
		Complementing the results from \cite{recht2012beneath}, we find upper bounds for C(d,n) under additional assumptions. Moreover, using free probability, we show that $C(d, n)  > 1$, thereby disproving the most optimistic conjecture from \cite{recht2012beneath}.
		  We also prove a deviation result for the symmetrized-AGM inequality which shows that the symmetric inequality almost holds for many classes of random matrices. Finally we apply our results to the incremental gradient method(IGM).

\section{introduction}
 \mbox{} \hspace{4mm} The arithmetic-geometric mean (AGM) inequality is of fundamental importance in mathematical analysis.  Augustin Cauchy (1789-1857) was the first to prove AGM in 1821. Then Liouville, Hurwitz, Steffensen, Bohr, Riesz,
 Sturm, Rado, Hardy, Littlewood, and Polya offered alternative proofs of the AGM inequalities in the same year. Much more recently,
  R\'{e} and Recht in \cite{recht2012beneath} realize that sampling without-replacement outperforms  sampling with replacement provided a noncommutative version of the arithmetic-geometric mean(AGM) inequality holds. They formulate several conjectures on AGM with connection to machine learning such as the incremental gradient method, empirical risk minimization and online learning. In particular, their proof, which employed the classical MacLaurin inequalities, led to improved convergence rate of incremental gradient method (IGM). \vspace{0.2in}\\ 
 Let us recall the famous MacLaurin inequalities for positive real numbers $x_1,...,x_n$ and the normalized $d$-th symmetric sums as 
 \begin{equation*}
 S_d=\binom{n}{d}^{-1} \sum_{\substack{\tau\subset \{1,...,n\}\\ |\tau|=d}}\prod_{i\in \tau }x_i.
 \end{equation*}
 where $ 1\le d\le n $ and $|\tau|:=$ the cardinality of $\tau$. According to the MacLaurin inequalities, we have
 \begin{equation*}
 S_1\geq \sqrt[2]S_2\geq \sqrt[3]S_3\geq ...\geq \sqrt[n]S_n \pl .
 \end{equation*}
 In particular,	$S_1\geq\sqrt[n]S_n$ is the standard AGM inequality. For more details about the classical AGM inequality see \cite{hardy1952inequalities}. 
\vspace{0.2in}

It's natural to ask whether the MacLaurin inequalities still hold if we replace the real numbers with positive definite matrices. Namely, if $A_1, \cdots, A_n$ are a collection of positive definite operators in some Hilbert space, the MacLaurin mean is defined as
\begin{equation}\label{1}
\mathcal{P}_d(\mathcal{A}) \lel \frac{(n-d)!}{n!}
\sum_{1\le j_1,...,j_d\le n \mbox{ \scriptsize all different}} 
A_{j_1}\cdots A_{j_d} \pl
\end{equation}
It is still open whether the norm of $\mathcal{P}_d$ is always less than that of $(\mathcal{P}_1)^n$, i.e.
\begin{equation}\label{01}
\|\mathcal{P}_d(\mathcal{A})\| \stackrel{?}{\le} C_1 \cdot \|\mathcal{P}_1(\mathcal{A})\|^d  
\end{equation}
 R\'{e} and Recht show that these noncommutative AGM inequalities hold when there are only two matrices or when all of the matrices commute. Arie, Felix and Rachel \cite{israel2016arithmetic} prove the inequality holds for products of up to three matrices by a variant of the classic Araki-Lieb-Thirring inequality.  R\'{e} and Recht also demonstrate that AGM holds for matrices with a constant depending on the degree $d$ and the dimension $m$ under some assumption. In \cite{2017arXiv170300546A},  we prove the \eqref{01} in both normed and ordered sense under some assumptions.  \\

  We define a symmetrized normalized d-th sums as follows:
  \begin{align}\label{02}
  \mathcal{S}_{d}(\mathcal{A}) =\frac{(n-d)!}{n!}\sum\limits_{1\le j_1,...,j_d\le n \mbox{ \scriptsize all different} }A_{j_{1}}^*A_{j_{2}}^*...A_{j_{d}}^*A_{j_{d}}...A_{j_{2}}A_{j_{1}}
  \end{align}
  $\mathcal{S}_d(\mathcal{A})$ is a second order polynormial. Each term in the summation form of $\mathcal{S}_d(\mathcal{A})$ is positive definite automatically. The symmetrized-AGM conjecture may be stated as  follows:
  \begin{align}\label{2}
  \|\mathcal{S}_d(\mathcal{A})\| \stackrel{?}{\le} C_2 \cdot  \|\mathcal{S}_1(\mathcal{A})\|^d.
  \end{align}
  However, in the symmetrized setting,  the following does not hold in general:
  $$\mathcal{S}_1(\mathcal{A} )^d  \neq \frac{1}{n^d} \sum_{j_1, \cdots, j_d }  A_{j_{1}}^*A_{j_{2}}^*...A_{j_{d}}^*A_{j_{d}}...A_{j_{2}}A_{j_{1}}.$$
  Therefore R\'{e} and Recht define the with-replacemnt expectation is defined as  
  \begin{equation}
  \mathbb{E}_{\text{wr}, k}[ x_{i_1}, \cdots, x_{i_n} ] = n^{-k} \sum_{j_1, \cdots, j_k}x_{j_1}^* \cdots x_{j_k}^*x_{j_k} \cdots x_{j_1}.
  \end{equation}
  That is, we average the value of f over all ordered tuples of elements from $(x_1, \cdots, x_n)$. Similarly,  the without-replacement expectation as 
  \begin{equation}
  \mathbb{E}_{\text{wo}, k}[x_{i_1}, \cdots, x_{i_n} ] = \frac{(n-k)!}{n!} \sum_{j_1 \neq j_2 \neq \cdots \neq j_k} x_{j_1}^* \cdots x_{j_k}^*x_{j_k} \cdots x_{j_1}.
  \end{equation}
 R\'{e} and Recht also asked for norm or order inequalities of the following:
\begin{align}\label{rrconj}
 \|\mathbb{E}_{wo, k}(\mathcal{A})\| &\stackrel{?}{\le} C_3 \cdot  \|\mathbb{E}_{wr, k}(\mathcal{A})\|, \\
 \label{oconj}
 \mathbb{E}_{wo, k}(\mathcal{A}) & \stackrel{?}{\le} C_4 \cdot  \mathbb{E}_{wr, k}(\mathcal{A}).
\end{align} 
 The norm symmetric AGM conjecture asks for \eqref{rrconj} with $C_3 = 1$. This still remains open. 
 \begin{theorem}
 	The order symmetric AGM conjecture \eqref{oconj} with $C_4 = 1$ is false for $n \geq 3, d= 3$.
 \end{theorem}

  Therefore, the problem of computating the best constant in \eqref{2} is particularly interesting. We explore the relation between the sequence $\mathcal{A}$, the length d and $C_2$ as follows:

\begin{theorem}
	Let $\{A_{i}\}$ be a family of operators on $B(H)$ satisfying  $ \frac{1}{n} \sum_{j= 1}^n A_{j}^*A_{j}=\mathds{1}.$ Then 
	\begin{enumerate}
		\item $\mathcal{S}_1 (\mathcal{A})= \mathds{1};$
		\item $(1-\varepsilon) \cdot \mathcal{S}_1(\mathcal{A})^d\leq \mathcal{S}_d(\mathcal{A})  \leq (1+ \varepsilon) \cdot \mathcal{S}_1(\mathcal{A})^d,$
	\end{enumerate}
	where $\varepsilon= \frac{1+ \sup\|A_{k}^*A_{k}\|}{n}\frac{d(d-1)}{2}, 1 \leq d \leq n$.
\end{theorem}

Under the condition $\frac{1}{n}\sum A_j^* A_j = \mathds{1}$, we obtain
$$\mathcal{S}_1(\mathcal{A})^d = \mathbb{E}_{wo,d}(\mathcal{A} ) \text{ and } \mathcal{S}_d(\mathcal{A}) = \mathbb{E}_{wr, d}(\mathcal{A} ).$$
Therefore, we also provide a proof of R\'{e} and Recht's conjecture with $C_3 = 1 + \epsilon$ in \eqref{rrconj}. However for random matrices, the condition $\frac{1}{n}\sum A_j^* A_j = \mathds{1}$ is no longer satisfied. Therefore we provide a deviation version of SAGM:
\begin{theorem} 
	Let $\{A_{i}\}$ be a family of i.i.d. random operators such that $E(A_{i}^*A_{i})=\mathds{1}$. If
	$ \big(\int \|\sum{A_{i}^* A_{i}}-E\Big(\sum{A_{i}^*A_{i}} \Big)  \|_{B(H)}^p dw \big)^{1/p} \leq n\varepsilon , d \ll n$, then  
	\begin{enumerate}
		\item   $  \Big(\int \| \mathbb{E}_{wo,d}(\mathcal{A})- E( \mathbb{E}_{wo,d}(\mathcal{A})) \|_{B(H)}^p dw \Big)^{1/p} \approx O(d\varepsilon), $
		\item $    \Big( \int \| \mathbb{E}_{wo,d}(\mathcal{A}) \|_{B(H)}^p dw \Big)^{1/p}  \leq ( 1+ O(d\varepsilon)    )          
		 \Big(  \int \|\mathbb{E}_{wr,d}(\mathcal{A})  \|_{B(H)}^p dw   \Big)^{1/p}$
	\end{enumerate}
\end{theorem}


The organization of the paper is as follows: we introduce some basic knowledge about combinatorial theory and von Neumann algebras. In section 2, we try to use some new techniques to estimate the symmetric version of AGM. After that, we give a counter example of the SAGM inequality using free probability. We give the deviation version of SAGM  at the end of this section. In section 3, we focus on  the  Incremental Gradient Method. We also construct three random matrices examples to satisfy the IGM with good constants using group representation theory and spherical design.  
\section{Upper bounds for symmetrized arithmetic-geometric mean inequalities}
We need to recall some definitions from the combinatorial theory of partitions  in \cite{andrews1998theory} and  \cite{rota1964foundations}. 
\begin{defi}
	Let $\mathbb{P}_n$ be the lattice of all the partitions of $\{1,...,n\}$. For two partitions $\sigma$ and $\pi$, we write $\sigma \leq \pi $ if every block of the partition $\sigma$ is contained in some block of $\pi$ (i.e., any block of the partition of $\pi$ can be written as a union of blocks of $\sigma$). In other words, $\pi$ is a refinement of $\sigma$. 
\end{defi}

\begin{rem}
	There are two trivial partitions, $\dot{0}$ and $\dot{1}$, where $\dot{0}$ is the partition into $n$ singletons and $\dot{1}$ is the partition of a single block. Namely, $\dot{0}$ is the smallest partition and $\dot{1}$ is the biggest one in $\mathbb{P}_n.$  
\end{rem}

\begin{rem}
	For a partition $\pi$,  let $\nu({\pi})$ be the number of the blocks of the partition $\pi$.
\end{rem}

For fixed $d$ we consider the following average symmetric product of noncommutative operators of length $d$: $$\mathcal{S}_d(\mathcal{A})=\frac{(n-d)!}{n!}\sum\limits_{\langle i_{1},...,i_{d}\rangle=\dot{0}}A_{i_{1}}^*A_{i_{2}}^*...A_{i_{d}}^*A_{i_{d}}...A_{i_{2}}A_{i_{1}}.$$

\begin{exam}
Let $A_{i}$ be a family of operators such that $\frac{1}{n}\sum A_{i}^*A_{i}= \mathds{1}.$ 
\begin{enumerate}
	\item For the partition $[13,2]$ where $d=3$ we can have the following upper-bound for the symmetric case of self-adjoint operator.
	%
	\begin{align*}
	&\|\sum\limits_{\langle i_{1},i_{2},i_{3}\rangle=\dot{0}} e_{kj,1}\otimes A_{k}A_{j}A_{k}\|^{2}
	=\|\sum e_{kj,k'j'}\otimes A_{k}^*A_{j}^*A_{k}^*A_{k}A_{j}A_{k}\|\\
	&=\|[\sum\limits_{k} e_{kk}\otimes 1 \otimes 1\otimes A_{i_{k}}^*] \times [\sum\limits_{j} 1\otimes e_{j1}\otimes 1\otimes A_{i_{j}}^*]\times[ \sum\limits_{k}e_{k1}\otimes 1 \otimes 1\otimes A_{i_{k}}^*] \\
	&\times  [ \sum\limits_{k'}e_{1k'}\otimes 1 \otimes 1\otimes A_{i_{k'}}] \times[\sum\limits_{j'} 1\otimes e_{1j'}\otimes 1\otimes A_{i_{j}}] \times  [\sum\limits_{j} 1\otimes e_{k'k'}\otimes 1\otimes A_{i_{k}}]\|
	\end{align*}
	The column term or row term is bounded by $\sqrt{n} $, and the diagonal term is bounded by $\sup \|A_{i_kk}\|$.  Therefore we have
	\begin{equation*}
	\|\sum\limits_{\langle i_{1},i_{2},i_{3}\rangle=\dot{0}} e_{kj,1}\otimes A_{k}A_{j}A_{k}\|^{2}   \leq \sup\|A_{i_{kk}}\|(\sqrt{n})^{4}\sup\|A_{i_{kk}}\|
	= n^{2} \sup\|A_{i_{kk}}^* A_{i_{kk}}\| .
	\end{equation*}
	\item  If we have the partition $[12,34]$ where $d=4$ then we have 
	\begin{align*}
	&\|\sum\limits_{\langle i_{1},i_{2},i_{3},i_{4}\rangle=\dot{0}} e_{kj,1}\otimes A_{k}A_{k}A_{j}A_{j}\|^{2}
	=\|\sum\limits_{\langle i_{1},i_{2},i_{3},i_{4}\rangle=\dot{0}} e_{kj,k'j'}\otimes A_{k}A_{k}A_{j}A_{j}A_{j}A_{j}A_{k}A_{k}\|\\
	&=\|\sum\limits_{k} 1\otimes e_{kk}\otimes A_{i_{k}}\| \times  \|\sum\limits_{k} 1\otimes e_{k1}\otimes A_{i_{k}}\|  \times \|\sum\limits_{j}e_{jj}\otimes 1\otimes A_{i_{j}}\|\times\\
	&\times\underbrace{\|\sum\limits_{j}e_{j1}\otimes 1\otimes A_{i_{j}}\| \times \|\sum\limits_{j'}e_{1j'}\otimes 1\otimes A_{i_{j'}}\|}_{\text{the middle term}}\times\\
	&\times \|\sum\limits_{j'}e_{j'j'}\otimes 1\otimes A_{i_{j'}}\| \times\|\sum\limits_{k'}1\otimes e_{1k'}\otimes A_{i_{k'}}\| \times \|\sum\limits_{k'}1\otimes e_{k'k'}\otimes A_{i_{k'}}\|
	\end{align*}
	Since we have two column terms, two row terms and four diagonal terms, then the symmetric summation has 
	\begin{equation*}
	\|\sum\limits_{\langle i_{1},i_{2},i_{3},i_{4}\rangle=\dot{0}} e_{kj,1}\otimes A_{k}A_{k}A_{j}A_{j}\|^{2} \leq n^2\cdot \sup\|A_{ik}A_{ik}^*\|^2
	=n^{|\sigma|}\cdot \sup\|A_{ik}A_{ik}^*\|^{d-|\sigma|}. 
	\end{equation*}
\end{enumerate}

\end{exam}
\vspace{0.1in} 
From the above examples we can generalize the upper bound for any partition. By Pisier's tensor techniques mentioned in \cite{pisier2000inequality},
we first need to embed each element $a_{i_k} \in M $ in larger space $(\otimes B(H)) \otimes M$. For each $i_{k}\in B_{i}$ we have  
\begin{displaymath}
Z_{i_{k}}=\left\{
\begin{array}{cc}
	1 \otimes \cdots \otimes e_{kk}\otimes \cdots \otimes 1\otimes A_{i_{k}}&  if~~i_{k}\neq \min B_{i}\\
    1 \otimes \cdots \otimes  e_{1k}\otimes \cdots \otimes 1\otimes A_{i_{k}}& if~~i_{k}=\min B_{i}\\
\end{array} \right.	
\end{displaymath}
where $\min B_{i}$ means the smallest index number and $e_{kk}, e_{1k}$ are matrix units in the $i^{th}$ component of the tensor form.

Similarly define the adjoint of the operator $Z_{i_{k}}$, denoted as $Z_{i_{k}}^*$, as follow:
\begin{displaymath}
Z_{i_{k}}^*=\left\{
\begin{array}{cc}
	1 \otimes \cdots \otimes e_{kk}\otimes \cdots \otimes 1\otimes A_{i_{k}}^*&  if~~i_{k}\neq \min B_{i}\\
    1 \otimes \cdots \otimes  e_{k1}\otimes \cdots \otimes 1\otimes A_{i_{k}}^*& if~~i_{k}=\min B_{i}\\
\end{array} \right.	
\end{displaymath}

\begin{lemma}\label{lem22}
Let $A_{i}$ be a family of operators such that $\frac{1}{n}\sum A_{i}^*A_{i}=\mathds{1}$. Denote $\tilde{Z}_{i_k} = \sum Z_{i_k},$ then we have 
\begin{displaymath}
	\|\tilde{Z}_{i_k}\| \leq \left\{
    \begin{array}{cc}
    \sup\limits_{m} \|A_m\|& if~~i_{k}\neq \min B_{i} \\ 
    \sqrt{n}  &  if~~i_{k}=\min B_{i}\\
    \end{array}\right.
\end{displaymath}
and consequently $ \tilde{Z}_{i_k}^* := \sum Z_{i_k}^*$, then we have	$\|\tilde{Z}_{i_k}^* \| = \|\tilde{Z}_{i_k} \|.$
\end{lemma}
\begin{proof}
If $i_k = \min B_{i},$ it turns out that $ \tilde{Z}_{i_k}$ generates a row space. Therefore,
\begin{align*}
 \|\tilde{Z}_{i_k} \| &= \| \tilde{Z}_{i_k} \tilde{Z}_{i_k}^* \|^{1/2}
 =  \| (\sum Z_{i_k}) (\sum Z_{i_k}^* )   \|^{1/2} \\
 &= \| 1 \otimes \cdots \otimes 1 \otimes (\sum A_{i_k}A_{i_k}^*)  \|^{1/2} = \|\sum A_{i_k}A_{i_k}^* \|^{1/2} = \sqrt{n} 
\end{align*}
If $i_k \neq \min B_{i},$ it turns out that $ \tilde{Z}_{i_k}$ generates a matrix form with diagonal component with elements $\{A_{i_k}\}$.
\begin{equation*}
\|\tilde{Z}_{i_k} \| = \| \sum Z_{i_k}  \| = \sup\limits_{i_{k}} \|A_{i_{k}} \|.   \qedhere
\end{equation*}
\end{proof}

The following is the key lemma for our main result in this section.
\begin{lemma}\label{L_{1}}
Let $\sigma$ be a partition in $\mathbb{P}_j$ with the blocks $B_{1},...,B_{\nu(\sigma)}$ then we have 
$$\|[\sigma]\|\leq n^{\nu(\sigma)}\cdot \sup\|A_{k}^*A_{k}\|^{|\sigma|-
\nu(\sigma)}.$$
\end{lemma} 
\begin{proof}
Let $\sigma=B_{1},B_{2},...,B_{\nu(\sigma)}$ where $B_{i}$'s are the blocks in this partition. 
Then we have 
 \begin{align*}
  &\|\sum\limits_{\langle i_{1},...i_{j}\rangle=\sigma}A_{i_{j}}^*A_{i_{j-1}}^*...A_{i_{1}}^*A_{i_{1}}...A_{i_{j-1}}A_{i_{j}}\|\\
  &=\|\sum\limits_{i_{1},...,i_{j}}Z_{i_{j}}^*Z_{i_{j-1}}^*...Z_{i_{1}}^*Z_{i_{1}}...Z_{i_{j-1}}Z_{i_{j}}\|\\
  &=\|\sum\limits_{i_{j}}Z_{ij}^*\sum\limits_{i_{j-1}}Z_{i_{j-1}}^*...\sum\limits_{i_{1}}Z_{i_{1}}^*\sum\limits_{i_{1}}Z_{i_{1}}...\sum\limits_{i_{j-1}}Z_{i_{j-1}}\sum\limits_{i_{j}}Z_{i_{j}}\| \\
 & = \|    \tilde{Z}_{i_j}^* \cdots \tilde{Z}_{i_1}^* \tilde{Z}_{i_1} \cdots \tilde{Z}_{i_j}     \| \leq  \prod_{k=1}^j \| \tilde{Z}_{i_k}  \|^2
  \leq n^{\nu(\sigma)}\cdot \sup\|A_{k}^*A_{k}\|^{j-
\nu(\sigma)}.
 \end{align*}
The last inequality comes from the observation of Lemma \ref{lem22}. It explains if we have a block, then there is only one operator in this block with index gives the upper bound norm equal to $n$ and the remaining operators with different indices are bounded by $\sup \|A_k^*A_k\|$.  Note that $\nu(\sigma)$ is the number of blocks. In the product form we have $\nu(\sigma)$ items give the bound $n$ and the remaining $j - \nu(\sigma)$ items give the bound $\sup \|A_k^*A_k\|$.
\end{proof}

\begin{rem}\label{rem24}
	In particular, we can observe that 
	\begin{enumerate}
		\item  Let $\sigma \text{ be } \dot{0}, ~\|[\dot{\sigma}]\| \leq n^k$; 
		\item Let $\sigma \text{ be } \dot{1}, ~\| [\dot{\sigma}]  \| \leq \sup \|A_k^* A_k\|^k$
	\end{enumerate}
	Here k means the cardinality of the set.
\end{rem}

In this paper without specific explaination, we always assume $\frac{1}{n}\sum A_{i}^*A_{i}= \mathds{1}, i.e. ~\mathcal{S}_1(\mathcal{A})  = \mathds{1}. $ Then
\begin{equation*}
\mathcal{S}_1(\mathcal{A})  - \mathcal{S}_d(\mathcal{A})=\frac{(n-d)!}{n!}\sum\limits_{\langle i_{1},...,i_{d}\rangle=\dot{0}}   (1- A_{i_{1}}^*A_{i_{2}}^*...A_{i_{d}}^*A_{i_{d}}...A_{i_{2}}A_{i_{1}} ). 
\end{equation*}
Most of the techniques we use in our first paper are not working for SAGM. So we need to introduce some other techniques. We need to introduce what Marius's method \textbf{ref} called the folding technique. Let's see the following example which explains this technique for $d=2$. 
\begin{align*}
1-A_{2}^*A_{1}^*A_{1}A_{2}&=1-A_{2}^*A_{2}+A_{2}^*A_{2}-A_{2}^*A_{1}^*A_{1}A_{2}\\
&=(1-A_{2}^*A_{2})+A_{2}^*(1-A_{1}^*A_{1})A_{2}
\end{align*}

Now for $d$ term we have 
\begin{align*}\nonumber
& 1-A_{d}^*...A_{1}^*A_{1}...A_{d}\\
&=(1-A_{d}^*A_{d})+A_{d}^*(1-A_{d-1}^*A_{d-1})A_{d}+A_{d}^*A_{d-1}^*(1-A_{d-2}^*A_{d-2})A_{d-1}A_{d}+...\\
&+ A_{d}^*A_{d-1}^*...A_{2}^*(1-A_{1}^*A_{1})A_{2}...A_{d-1}A_{d}\\
&=\sum_{j=1}^{d}A_{d}^*...A_{j+1}^*(1-A_{j}^*A_{j})A_{j+1}...A_{d}
\end{align*}

Given a partition $\sigma$ , and $i_1$ is not a singleton in $\sigma$. We denote $$[[\sigma]] : = \sum\limits_{\langle i_{1},...,i_{j}\rangle=\sigma} A_{i_{j}}^*...A_{i_{2}}^*(1-A_{i_{1}}^*A_{i_{1}})A_{i_{2}}...A_{i_{j}}.$$

\begin{lemma}
Given a partition $\sigma \in \mathbb{P}_j$ and $i_1$ is not a singleton element in the partition $\sigma$, then 
\begin{equation}\label{eq21}
\| [[\sigma]] \| \leq  n^{\nu(\sigma)}\cdot C^{|\sigma|-\nu(\sigma)}(1+ \frac{1}{C}),
\end{equation}
where $C =\sup\|A_{k}^*A_{k}\|.$
\end{lemma}
\begin{proof}
Let $\gamma $ be the partition without index $i_1$,  then we know $\nu(\sigma) = \nu(\gamma)$, $|\gamma| = j-1$ and $|\sigma|= j.$
\begin{align*} 
[[\sigma]] &= \sum\limits_{\langle i_{1},...,i_{j}\rangle=\sigma} A_{i_{j}}^*...A_{i_{2}}^*(1-A_{i_{1}}^*A_{i_{1}})A_{i_{2}}...A_{i_{j}} \\ 
&= \sum\limits_{\langle i_{1},...,i_{j}\rangle=\sigma} A_{i_{j}}^*...A_{i_{2}}^*A_{i_{2}}...A_{i_{j}} -  \sum\limits_{\langle i_{1},...,i_{j}\rangle=\sigma} A_{i_{j}}^*...A_{i_{1}}^*A_{i_{1}}...A_{i_{j}} \\
& = \sum\limits_{\langle i_2 ...,i_{j}\rangle=\gamma} A_{i_{j}}^*...A_{i_{2}}^*A_{i_{2}}...A_{i_{j}} -  \sum\limits_{\langle i_{1},...,i_{j}\rangle=\sigma} A_{i_{j}}^*...A_{i_{1}}^*A_{i_{1}}...A_{i_{j}} \\
& = [\gamma] - [\sigma]
\end{align*}
Then we consider the norm 
\begin{align*}
\|[[\sigma]]\| \leq \|[\gamma]\| + \|[\sigma]\| & \leq   n^{\nu(\gamma)}\cdot \sup\|A_{k}^*A_{k}\|^{|\gamma|-
\nu(\gamma)}+ n^{\nu(\sigma)}\cdot \sup\|A_{k}^*A_{k}\|^{|\sigma|-\nu(\sigma)} \\
&\leq n^{\nu(\sigma)}\cdot \sup\|A_{k}^*A_{k}\|^{|\sigma|-\nu(\sigma)}\Big(1+ \frac{1}{   \sup\|A_{k}^*A_{k}\|  }\Big).\qedhere
\end{align*} 
\end{proof}

\begin{theorem}
Let $\{A_{i}\}$ be a family of operator's on $B(H)$ satisfying $\frac{1}{n}\sum A_{i}A_{i}^*= \mathds{1}.$ Then we have 
\begin{equation}\label{bdd1}
\| \mathds{1}- \mathcal{S}_d(\mathcal{A})\|\leq \frac{(1+C)}{n}\frac{d(d-1)}{2}
\end{equation}
where $C =\sup\|A_{k}^*A_{k}\|.$
\end{theorem}
\begin{proof}
 Let $i_{j}=l$ and define $n_{l}:=[1,...,n]-{l}$ and $n_{l}^{j-1}:=[1,...,n]_{j-1}-{l}.$ Also, let $t$ represents the indices such that if $$t\notin [n_{l}]\Rightarrow~~~\text{some of the indices of t are equal to}~l.$$ 
 Then we have 
 \begin{align}
  \mathds{1}- \mathcal{S}_d(\mathcal{A})=\sum\limits_{j=1}^{d}\frac{(n-j)!}{n!}\underbrace{\sum\limits_{\langle i_{1},...,i_{j}\rangle=\dot{0}} A_{i_{1}}^*...A_{i_{j-1}}^*(1-A_{i_j}^*A_{i_j})A_{i_{j-1}}...A_{i_{1}}}_{I_{j}}
 \end{align}
Then let $l = i_j$
 \begin{align*}
 & I_{j}=\sum\limits_{\substack{\langle i_{1},...,i_{j}\rangle=\dot{0}\\ i_{1}\neq i_{2}\neq...\neq i_{j}}}A_{i_{1}}^*...A_{i_{j-1}}^*(1-A_{l}^*A_{l})A_{i_{j-1}}...A_{i_{1}}\\
 &=\sum\limits_{i_{1}}...\sum\limits_{i_{j-1}}\sum\limits_{l=1}^{n} A_{i_{1}}^*...A_{i_{j-1}}^* (1-A_{l}^*A_{l})A_{i_{j-1}}...A_{i_{1}}\\
 &=\sum\limits_{l=1}^{n}\sum\limits_{\substack{\langle i_{1},...,i_{j-1}\rangle=\dot{0}\\i_{1},...,i_{j-1}\in n_{l}^{j-1}}}A_{i_{1}}^*...A_{i_{j-1}}^*(1-A_{l}^*A_{l})~A_{i_{j-1}}...A_{i_{1}}
 \end{align*}
 Then by distributing the sum for both terms and from the given assumption we have 
 \raggedbottom
 \begin{align*}
 & I_{j}=\underbrace{\sum\limits_{\substack{\langle i_{1},...,i_{j-1\rangle}=\dot{0}\\ 1\leq i_{1},...,i_{j-1}\leq n}}A_{i_{1}}^*...A_{i_{j-1}}^* \sum\limits_{l=1}^{n} (1-A_{l}^*A_{l})~A_{i_{j-1}}...A_{i_{1}}}_{=0}\\
 & - \sum\limits_{l=1}^{n}\sum\limits_{\substack{\langle i_{1},...,i_{j-1\rangle}=\dot{0}\\ i_{1},...,i_{j-1} \text{not all of them in } n_{l}^{j-1}}} A_{i_{1}}^*...A_{i_{j-1}}^* (1-A_{l}^*A_{l}) A_{i_{j-1}}...A_{i_{1}}\\
 & = \underbrace{-\sum\limits_{l=1}^{n}\sum\limits_{\substack{\langle i_{1},...,i_{j-1}\rangle=\dot{0}\\i_{1},...,i_{j-1}\text{not all of them in } n_{l}^{j-1}}} A_{i_{1}}^*...A_{i_{j-1}}^* A_{i_{j-1}}...A_{i_{1}}}_{\alpha}\\
 & + \underbrace
 {\sum\limits_{l=1}^{n}\sum\limits_{\substack{\langle i_{1},...,i_{j-1}\rangle=\dot{0} \\i_{1},...,i_{j-1}\text{not all of them in } n_{l}^{j-1}}} A_{i_{1}}^*...A_{i_{j-1}}^*A_{l}^*A_{l}A_{i_{j-1}}...A_{i_{1}}}_{\beta} \label{m1}
 \end{align*}

 Notice that the sum run over the restricted partition $\langle i_{1},...,i_{j-1}\rangle=\dot{0}$ with condition that $\text{not all of them in } n_{l}^{j-1}$. This will force just one $i_{k}=l,$ so we have $(j-1)$ choices.
 \raggedbottom
 \begin{align*}
 &\|\mathds{1}- \mathcal{S}_d(\mathcal{A})\| \leq \sum\limits_{j=1}^{d}\frac{(n-j)!}{n!}(j-1)\|\sum\limits_{\substack{\langle i_{1},...,i_{j-1}\rangle=\dot{0}\\ i_{k}=l}} A_{i_{1}}^*...A_{i_{j-1}}^* A_{i_{j-1}}...A_{i_{1}}\|\\ &+\sum\limits_{j=1}^{d}\frac{(n-j)!}{n!}(j-1)\|\sum\limits_{\substack{\langle i_{1},...,i_{j-1}\rangle=\dot{0} \\i_{k}=l}} A_{i_{1}}^*...A_{i_{j-1}}^* A_{l}^*A_{l}A_{i_{j-1}}...A_{i_{1}}\|\\ 
 \end{align*}
 
 \vspace{-0.2in}
  Now using the fact that the norm of the restricted partition is less than the full one, i.e. $\|\langle\sigma \rangle\|\leq \|[\sigma]\|.$
 \vspace{-0.2in}
 \begin{align*}
  \|\mathds{1}- \mathcal{S}_d(\mathcal{A})\|&\leq  \sum\limits_{j=1}^{d}\frac{(n-j)!}{n!}(j-1)\|\sum\limits_{\substack{[ i_{1},...,i_{j-1}]=[\dot{0}]\\i_{k}=l}} A_{i_{1}}^*...A_{i_{j-1}}^* A_{i_{j-1}}...A_{i_{1}}\|\\
 &+C\sum\limits_{j=1}^{d}\frac{(n-j)!}{n!}(j-1) \|\sum\limits_{\substack{[ i_{1},...,i_{j-1}]=[\dot{0}] \\i_{k}=l}} A_{i_{1}}^*...A_{i_{j-1}}^*A_{i_{j-1}}...A_{i_{1}}\|
 \end{align*} 
 By Remmark \ref{rem24} for the partition $\dot{0}$, we obtain 
 \begin{align*}
 \|\mathds{1}- \mathcal{S}_d(\mathcal{A})\|&\leq \sum\limits_{j=1}^{d}\frac{(n-j)!}{n!}(j-1) n^{j-1}+ C\sum\limits_{j=1}^{d}\frac{(n-j)!}{n!}(j-1)n^{j-1}\\
 &\leq \sum\limits_{j=1}^{d} \frac{(j-1)}{n}+\frac{C(j-1)}{n}=\frac{(1+C)}{n}\frac{d(d-1)}{2}\qedhere
 \end{align*} 
 \end{proof}
 
 Note that if the condition $\frac{1}{n} \sum A_i^*A_i = \mathds{1}$ doesn't hold,  we get
 $$\mathcal{S}_1(\mathcal{A})^d = \Big(\frac{1}{n} \sum A_i^*A_i \Big)^d \neq \frac{1}{n^d} \sum_{j_1, \cdots, j_d }  A_{j_{1}}^*A_{j_{2}}^*...A_{j_{d}}^*A_{j_{d}}...A_{j_{2}}A_{j_{1}}.$$

 To formulate a sufficient analysis, let us first formalize some notation mentioned in \cite{recht2012beneath}. Throughtout in the remaining of this paper, $[n]$ denotes the set of integers from 1 to n. Let $\mathbb{D}$ be some domain, $f: \mathbb{D}^k \rightarrow \mathbb{R}$, and $\mathcal{A} := (x_1, \cdots, x_n)$ a set of n elements from $\mathbb{D}$. We define the without-replacement expectation as 
 \begin{equation}\label{wo}
 \mathbb{E}_{\text{wo}, k}[f( \mathcal{A})] = \frac{(n-k)!}{n!} \sum_{j_1 \neq j_2 \neq \cdots \neq j_k} f(x_{j_1}, \cdots, x_{j_k}).
 \end{equation}
 That is, we average the value of f over all ordered tuples of elements from $(x_1, \cdots, x_n)$. Similarly, the with-replacemnt expectation is defined as  
 \begin{equation*}
 \mathbb{E}_{\text{wr}, k}[f(\mathcal{A} )] = n^{-k} \sum_{(j_1, \cdots, j_k) = \dot{1}} f(x_{j_1}, \cdots, x_{j_k}).
 \end{equation*}
 \begin{rem}
 	With the above definition \eqref{wo}, we take a specific function defined as  $f(x_{j_1}, \cdots, x_{j_k}) : = \prod\limits_{j=1}^{k}x_{i_{k-j+1}}  \prod\limits_{j=1}^k x_{i_j}$. Then  $\mathbb{E}_{\text{wo}, k}[f( \mathcal{A} )] = \mathcal{S}_k(\mathcal{A})$.
 \end{rem}
 \raggedbottom
 Next let's provide a deviation version of SAGM as follows:
 \begin{theorem} \label{derivation}
 	Let $\{A_{i}\}$ be a family of i.i.d. random operators such that $E(A_{i}^*A_{i})=\mathds{1}$. If
 	$\vertiii{\sum{A_{i}^* A_{i}}-E\Big(\sum{A_{i}^*A_{i}} \Big)  } \leq n\varepsilon , d \ll n$, then 
 	\begin{enumerate}
 		\item   $\vertiii{   \mathbb{E}_{wo,d}(\mathcal{A})- E( \mathbb{E}_{wo,d}(\mathcal{A})) } \approx O(d\varepsilon), $
 		\item $\vertiii{   \mathbb{E}_{wo,d}(\mathcal{A})} \leq (1+  O(d\varepsilon)) \vertiii{\mathbb{E}_{wr,d}(\mathcal{A})}$
 	\end{enumerate}
 \end{theorem}
 \begin{proof}
(1)	Since $\{A_{i}\}$ are  i.i.d., $E( \mathbb{E}_{wo,d})  = \mathds{1}$. Therefore 
 		\begin{align}\nonumber 
 		&\vertiii{\mathbb{E}_{wo,d}(\mathcal{A})-   E( \mathbb{E}_{wo,d})  } = \vertiii{ \mathds{1} -  \mathbb{E}_{wo,d}(\mathcal{A})} \leq    \\
 		& \leq \sum\limits_{j=1}^{d}\frac{(n-j)!}{n!} \sum_{l \in [n]}  \vertiii{\sum\limits_{\substack{\langle i_{1},...,i_{j-1}\rangle=\dot{0}\\ i_{j}=l , i_k \in [n], k \neq j  }} A_{i_{1}}^*...A_{i_{j-1}}^* (1-A_{i_j}^*A_{i_j})A_{i_{j-1}}...A_{i_{1}}}\\ 
 		&+\sum\limits_{j=1}^{d}\frac{(n-j)!}{n!} \sum_{l \in [n]} \vertiii{\sum\limits_{\substack{\langle i_{1},...,i_{j-1}\rangle=\dot{0}\\ i_{j}=l, i_k \in [n_l], k \neq j }} A_{i_{1}}^*...A_{i_{j-1}}^* A_{i_{j-1}}...A_{i_{1}} } \\
 		&+\sum\limits_{j=1}^{d}\frac{(n-j)!}{n!} \sum_{l \in [n]} \vertiii{\sum\limits_{\substack{\langle i_{1},...,i_{j-1}\rangle=\dot{0}\\ i_{j}=l, i_k \in [n_l], k \neq j }} A_{i_{1}}^*...A_{i_{j}}^* A_{i_{j}}...A_{i_{1}}  }
 		\end{align}
 		
 		The term on the left side of the inequality (2.4) can factor out $\sum(1 - A_{i_j}^* A_{i_j}) $ since the index $i_j$ is independent with the others. And the norm in (2.5) is invariant with respect to the index $l$. (2.6)  is different with (2.5) with a constant C, after factoring out the $\|A_{i_j}^* A_{i_j}\|$.
 		\raggedbottom
 		
 		\begin{align*}
 		&\vertiii{\mathbb{E}_{wo,d}(\mathcal{A})- E( \mathbb{E}_{wo,d}(\mathcal{A}))  } \\
 		&\leq  \sum\limits_{j=1}^{d}\frac{(n-j)!}{n!}       \vertiii{\sum\limits_{\substack{[ i_{1},...,i_{j-1}]=\langle \dot{0} \rangle \\i_{k} \in [n], k = 1, \cdots, j-1}} A_{i_{1}}^*...A_{i_{j-1}}^* A_{i_{j-1}}...A_{i_{1}}} \cdot  \vertiii{ \sum_{i_j} (1- A_{i_j}^* A_{i_j}) }\\
 		&+(1+C)\sum\limits_{j=1}^{d}\frac{(n-j)!}{n!}(j-1)   \vertiii{\sum\limits_{\substack{[ i_{1},...,i_{j-1}]=\langle \dot{0} \rangle \\i_{k} \in [n], k = 1, \cdots, j-1    }} A_{i_{1}}^*...A_{i_{j-1}}^*A_{i_{j-1}}...A_{i_{1}}}
 		\end{align*}
 		Therefore we get
 		
 		\begin{align*}
 		\vertiii{\mathbb{E}_{wo,d}(\mathcal{A})- E( \mathbb{E}_{wo,d}(\mathcal{A}))  }  \leq  \sum\limits_{j=1}^{d}\frac{(n-j)!}{n!}  \frac{n!}{(n-j+1)!} \vertiii{\mathbb{E}_{wo,j-1} (\mathcal{A}) }\cdot \Big(n\varepsilon + (1+C) \cdot (j-1)  \Big)  
 		\end{align*}
 		
 		Then by using $\mathbb{E}_{wo, j-1}(\mathcal{A})= \mathbb{E}_{wo, j-1}(\mathcal{A}) - E(  \mathbb{E}_{wo, j-1}(\mathcal{A}) ) +   E( \mathbb{E}_{wo, j-1}(\mathcal{A})) ,$ we have
 		\begin{align*}
 		\vertiii{ \mathbb{E}_{wo, d}(\mathcal{A})- E( \mathbb{E}_{wo, d}(\mathcal{A})) } &  \leq \sum_{j=1}^{d} \frac{ n\varepsilon 
 			+ (1+C) \cdot (j-1) }{n - j +1} \vertiii{  \mathbb{E}_{wo, j-1}(\mathcal{A})- E( \mathbb{E}_{wo, j-1}(\mathcal{A}))  }   \\
 		&+ \sum_{j=1}^{d} \frac{ n\varepsilon + (1+C) \cdot (j-1) }{n - j +1}
 		\end{align*}
 		
 		Using the discrete case of Gr\"{o}nwall's lemma \cite{emmrich1999discrete}, 
 		\begin{equation*}
 		\vertiii{ \mathbb{E}_{wo, d}(\mathcal{A})- E( \mathbb{E}_{wo, d}(\mathcal{A}))  } \leq f(d) \exp( f(d) ), f(d): =  \sum_{j=1}^d \frac{ n\varepsilon + (1+C) \cdot (j-1) }{n - j +1} 
 		\end{equation*}
 		
 		It remains to find an upper bound for f(d)
 		\begin{align*}
 		f(d) & = \sum_{j=1}^d \frac{ (n-j+1)\varepsilon + (\varepsilon+ 1+C) \cdot (j-1) }{n - j +1}  
 		= \sum_{j=1}^d (\varepsilon + (\varepsilon+1+C) \cdot \frac{j-1}{n-j+1}   )  \\
 		& \leq  d\varepsilon + (\varepsilon +1+C) \sum_{j=1}^d \frac{j-1}{n-j+1}  \leq d\varepsilon + (\varepsilon+1+C)\cdot \frac{d(d-1)}{n-d+1}
 		\end{align*}
 		Since $d \ll n, f(d) \approx O(d\varepsilon)$, then
 		$	\vertiii{ \mathbb{E}_{wo,d}(\mathcal{A})- E( \mathbb{E}_{wo, d}(\mathcal{A}))  } \approx O(d\varepsilon).$ \\
(2) Since $\{ A_i\}$ are i.i.d., $E(\mathbb{E}_{wr,d} ) = \mathds{1}.$ 
\begin{align}\nonumber
\mathbb{E}_{wr, d} - E(\mathbb{E}_{wr,d} )  &= \frac{1}{n^d}\sum_{j_1, \cdots, j_d}x_{j_1}^* \cdots x_{j_d}^*x_{j_d} \cdots x_{j_1}  - \mathds{1} \\\label{27eq}
& = \frac{1}{n^{d-1}}\sum_{j_1, \cdots, j_{d-1}}x_{j_1}^* \cdots x_{j_{d-1}}^*\Big(\frac{1}{n} \sum_{j_d} x_{j_d}^*x_{j_d}  -1\Big)  x_{j_{d-1}}  \cdots x_{j_1} \\\nonumber
&~~~~~~~~+ \frac{1}{n^{d-1}}  \sum_{j_1, \cdots, j_{d-1}}x_{j_1}^* \cdots x_{j_{d-1}}^*x_{j_{d-1}}  \cdots x_{j_1} - \mathds{1} 
\end{align}
Thanks to $\vertiii{\sum{A_{i}^* A_{i}}-E\Big(\sum{A_{i}^*A_{i}} \Big)  } \leq n\varepsilon $ and H\"{o}ld inequality, the term \eqref{27eq} admits the following estimate
\begin{equation*}
\vertiii{\sum_{j_1, \cdots, j_{d-1}}x_{j_1}^* \cdots x_{j_{d-1}}^*\Big(\frac{1}{n} \sum_{j_d} x_{j_d}^*x_{j_d}  -1\Big)  x_{j_{d-1}}  \cdots x_{j_1} }  
	\leq   
	\varepsilon \vertiii{\sum_{i_1 = 1}^n x_{i_1}^*x_{i_1}}^{d-1}  
\end{equation*}
Since $\vertiii{\sum_{i_1 = 1}^n x_{i_1}^*x_{i_1}} \leq n (1 + \varepsilon) $ and triangle inequality, we get 
$$ \vertiii{\mathbb{E}_{wr, d} - E(\mathbb{E}_{wr,d} )} \leq \frac{1}{n^{d-1}} \varepsilon \cdot \big( n (1 + \varepsilon)\big)^{d-1}+ \vertiii{\mathbb{E}_{wr, d-1} - E(\mathbb{E}_{wr,d-1} )  }$$
Denote $a_d = \vertiii{\mathbb{E}_{wr, d} - E(\mathbb{E}_{wr,d} )} ,$ we have a sequence $\{a_d\}$ has $a_1 \leq \varepsilon$ and the iteration inequality,
\begin{equation*}
a_d \leq \varepsilon ( 1+ \varepsilon)^{d-1}  + a_{d_1}
\end{equation*}
Take the iteration for d times, we get
\begin{equation} 
 \vertiii{\mathbb{E}_{wr, d} (\mathcal{A} )- E(\mathbb{E}_{wr,d} (\mathcal{A}  )) } \leq \sum_{i=1}^{d} \varepsilon ( 1+ \varepsilon)^{i-1}  \leq O(d\varepsilon)
\end{equation}
Then by triangle inequality, we obtain
\begin{align*}
&\vertiii{   \mathbb{E}_{wo,d}(\mathcal{A})-  \mathbb{E}_{wr,d}(\mathcal{A})}   
 =  \vertiii{    \mathbb{E}_{wo,d}(\mathcal{A})- E( \mathbb{E}_{wo,d}(\mathcal{A})  ) - \Big(\mathbb{E}_{wr,d}(\mathcal{A}) -E( \mathbb{E}_{wr,d}(\mathcal{A})  ) \Big) } \\
& \leq \vertiii{  \mathbb{E}_{wo,d}(\mathcal{A})- E( \mathbb{E}_{wo,d}(\mathcal{A})  )   } + \vertiii{\mathbb{E}_{wr,d}(\mathcal{A}) -E( \mathbb{E}_{wr,d}(\mathcal{A})  )} \approx O(d\varepsilon)\qedhere
\end{align*}
 \end{proof}

\section{A Counterexample For the Symmetric Arithmetic Geometric Mean Inequality}
 In this section, we provide an example from free probability \cite{voiculescu2000lectures} which proves that symmetric arithmetic geometric mean inequality (SAGM) is not true in general.
 
Let's recall the construction of the reduced amalgamated free product of von Neumann algebras. Let $A_{1},...,A_{n}$ be a family of von Neumann algebras and let $\langle a \rangle$ be a common von Neumann subalgebra of $A_{k}$ generated by an element $a.$ We will assume that there is a normal faithful conditional expectation $E_{k}:A_{k}\rightarrow \langle a \rangle$ for each $k.$ Let $A=\ast_{\langle a\rangle}A_{k}$ be the reduced amalgamated free product of $A_{1},...,A_{n}$ over $\langle a \rangle$ with respect to the $E_{k}.$ We are concerned about  the case when $d=3$. 

\begin{fact} \label{rem33}
	We list some properties for the von Neumann algbras  $A_{j}:=\langle u_{j},a\rangle$ generated by $u_j $ and a:
	\begin{enumerate}
		\item $A_{j}$ is freely independent over $A$,then $E_{A_{j}}(X)=E_{A}(X).$
		\item If $X=b_{1}b_{2}...b_{n}$ with $b_{1}\in \mathring{A_{j}}, b_{i}\in \mathring{A}$, then $E_{A}(X)=0$.
		\item  $E_{A}(u_{k}ab)=\tau(u_{k})ab=0$ for all  $a,b\in A.$ 
		\item $\tau(c_{j_1} \cdots c_{j_n} ) = 0, c_{j_i} \in \mathring{A}_{j_i} ~ j_i  \in \{1, \cdots n-1 \}, c_{j_n} \in A_n \text{ and } j_1 \neq j_2 \neq \cdots \neq  j_n .$
	\end{enumerate}
\end{fact}

 \begin{theorem}\label{counter}
 Let $u_{j}$ be unitaries and $\{ u_1, \cdots, u_n, a \}$ be freely independent operators such that $$a_{j}=au_{j},~\tau (u_j)=\tau(a)=0,~\tau(a^{2})=1~\text{and}~a^2 \neq 1.$$ Then 
 \begin{align*}
 \mathbb{E}_{wo,3}(a_{j_{1}},a_{j_{2}},a_{j_{3}})&=\frac{1}{n(n-1)(n-2)}\sum\limits_{ \langle  j_{1},j_{2},j_{3} \rangle = \dot{0} }a_{j_{1}}a_{j_{2}}a_{j_{3}}a_{j_{3}}^*a_{j_{2}}^*a_{j_{1}}^*\\
\mathbb{E}_{wr, 3}(a_{j_{1}},a_{j_{2}},a_{j_{3}})&=\frac{1}{n^{3}}\sum\limits_{ \langle j_{1},j_{2},j_{3} \rangle = \dot{1} }a_{j_{1}}a_{j_{2}}a_{j_{3}}a_{j_{3}}^*a_{j_{2}}^*a_{j_{1}}^*
 \end{align*} 
 do not satisfy 
 $$ \mathbb{E}_{wo,3} (a_{j_1}, a_{j_2}, a_{j_3}) \leq \mathbb{E}_{wr, 3}  (a_{j_1}, a_{j_2}, a_{j_3}).$$
 
 \end{theorem}
 
 In this following, we will replace $\mathbb{E}_{wo,3} $ with  $\mathbb{E}_{wo}$ (respectively replace $\mathbb{E}_{wr,3} $ with  $\mathbb{E}_{wr}$). First We need  to prove two main lemmas.
 \begin{lemma}\label{lem32}
 $\tau( \mathbb{E}_{wo}(a_{j_{1}},a_{j_{2}},a_{j_{3}}))=\tau( \mathbb{E}_{wr}(a_{j_{1}},a_{j_{2}},a_{j_{3}}))=\tau(a^{2})^3$
 \end{lemma}
 \begin{proof}
 We will use the free independent condition which says that $\tau(\mathring{a_{i_{1}}}\mathring{a_{i_{2}}}...\mathring{a_{i_{n}}})=0$ if $i_1 \neq i_2 \neq \cdots \neq i_n$ and $\mathring{a} := a - \tau(a)$. 
 We claim that for all choices $j_1, j_2, j_3$, we have
 \begin{equation}
 \tau(a_{j_{1}}a_{j_{2}}a_{j_{3}}a_{j_{3}}^*a_{j_{2}}^*a_{j_{1}}^*)=\tau(a^{2})^{3} 
 \end{equation}

 Let us start by using the definition of the operator $a_{i_{j}}$ frequently.
 \begin{align*}
  &\tau(a_{j_{1}}a_{j_{2}}\underbrace{a_{j_{3}}a_{j_{3}}^*}_{=a^{2}}a_{j_{2}}^*a_{j_{1}}^*)=\tau(a_{j_{1}}a_{j_{2}}a^{2}a_{j_{2}}^*a_{j_{1}}^*)\\
  &=\tau(a_{j_{1}}a_{j_{2}}\mathring{(a^2)}a_{j_{2}}^*a_{j_{1}}^*)+\tau(a_{j_{1}}a_{j_{2}}a_{j_{2}}^*a_{j_{1}}^*)\tau(a^2)\\
  &=\tau(a_{j_{1}}a_{j_{2}}\mathring{(a^2)}a_{j_{2}}^*a_{j_{1}}^*)+\tau(a_{j_{1}}a^{2}a_{j_{1}}^*)\tau(a^2)\\
  &=\tau(a_{j_{1}}a_{j_{2}}\mathring{(a^2)}a_{j_{2}}^*a_{j_{1}}^*)+\tau(\mathring{a^{2}}a_{j_{1}}^*a_{j_{1}})\tau(a^2))+\tau(a^2)\tau(a_{j_{1}}^*a_{j_{1}})\tau(a^2)\\
  &=\tau(a_{j_{1}}a_{j_{2}}\mathring{(a^2)}a_{j_{2}}^*a_{j_{1}}^*)+\underbrace{\tau(\mathring{a^{2}}u_{j_{1}}^*\mathring{a^2}u_{j_{1}})}_{=0}\tau(a^2)+\underbrace{\tau(\mathring{a^{2}}u_{j_{1}}^*u_{j_{1}})}_{=0}\tau(a^2)\tau(a^2)+\tau(a^2)\tau(a^2)\tau(a^2)\\
  &=\tau(a_{j_{2}}\mathring{(a^2)}a_{j_{2}}^*a_{j_{1}}^*a_{j_{1}})+\tau(a^2)\tau(a^2)\tau(a^2)\\
  &=\tau(a_{j_{2}}\mathring{(a^2)}a_{j_{2}}^*u_{j_{1}}^{*}a^{2}u_{j_{1}})+\tau(a^2)\tau(a^2)\tau(a^2)\\
  &=\underbrace{\tau(a_{j_{2}}\mathring{(a^2)}a_{j_{2}}^*u_{j_{1}}^{*}\mathring{a^{2}}u_{j_{1}})}_{=0}+\underbrace{\tau(u_{j_{2}}a\mathring{(a^2)}u_{j_{2}}^*a)}_{=0}\tau(a^2)+\tau(a^2)\tau(a^2)\tau(a^2)\\
  &=\tau(a^2)\tau(a^2)\tau(a^2)=\tau(a^2)^3=1.\qedhere
 \end{align*}
 \end{proof}
  The above zero terms hold thanks to freeness.  We will use the following folding techniques introduced by Junge in section 6 of \cite{junge2006operator} to write $\mathds{1}-\mathbb{E}_{wo}$ and $\mathds{1}-\mathbb{E}_{wr}$ for the operators $a_{j}$ where $a_{j}=au_{j}$. 
 Recall that 
 \begin{align*}
 \mathds{1}-a_{j_1}a_{j_2} a_{j_3} a_{j_3}^*a_{j_2}^*a_{j_1}^*&= \mathds{1}-a_{j_1}a_{j_1}^*+a_{j_1}( \mathds{1}-a_{j_2}a_{j_3}a_{j_3}^*a_{j_2}^*)a_{j_1}^*\\
 &=\mathds{1}-a_{j_1}a_{j_1}^*+a_{j_1}(\mathds{1} -a_{j_2}a_{j_2}^*)a_{j_1}^*+a_{j_1}a_{j_2}( \mathds{1}-a_{j_3}a_{j_3}^*)a_{j_2}^*a_{j_1}^*.
 \end{align*}  
 Then we have 
 \begin{align}\nonumber 
 \mathds{1}- \mathbb{E}_{wo}(a_{l},a_{j},a_{k})&=\frac{1}{n}\sum\limits_{j=1}^{n}1-a^2+\frac{1}{n(n-1)}\sum\limits_{j=1,j\neq k}^{n}a_{j}a_{k}a_{k}^*a_{j}^*\\\nonumber
 &+\frac{1}{n(n-1)(n-2)}\sum\limits_{j\neq k\neq l}a_{j}a_{k}(1-a_{l}a_{l}^*)a_{k}^*a_{j}^*\\\nonumber
 &=(1-a^2)+\frac{1}{n}\sum\limits_{j=1}^{n}a_{j}(1-a^{2})a_{j}^*+\frac{1}{n(n-1)}\sum\limits_{j\neq k}a_{j}a_{k}(1-a^2)a_{k}^*a_{j}^*\\\nonumber
 &=(1-a^2)+\frac{1}{n}\sum\limits_{j}a_{j}(1-a^2)a_{j}^*+\frac{1}{n(n-1)}\sum\limits_{j,k}a_{j}a_{k}(1-a^2)a_{k}^*a_{j}^*\\\
 &-\frac{1}{n(n-1)}\sum\limits_{j}a_{j}^2(1-a^2)(a_{j}^*)^{2}
 \end{align}
 and for the expectation for with-replacement
 \begin{align}
 \mathds{1}-\mathbb{E}_{wr}(a_{l},a_{j},a_{k})=(1-a^2)+\frac{1}{n}\sum\limits_{j=1}^{n}a_{j}(1-a^2)a_{j}^*+\frac{1}{n^2}\sum\limits_{j\neq k}a_{j}a_{k}(1-a^2)a_{k}^*a_{j}^*
 \end{align} 
 Then we take the difference between $1-\mathbb{E}_{wo}$ and $1- \mathbb{E}_{wr}$, then we have 
 \begin{align*}
 \mathbb{E}_{wo} - \mathbb{E}_{wr} &= \mathds{1}-\mathbb{E}_{wr}-(\mathds{1}- \mathbb{E}_{wo}) \\
& =\Big[\frac{1}{n^2}-\frac{1}{n(n-1)}\Big]\sum\limits_{j,k}a_{j}a_{k}(1-a^2)a_{k}^*a_{j}^*+\frac{1}{n(n-1)}\sum\limits_{j=1}^{n}a_{j}(a_{j}(1-a^{2})a_{j}^*)a_{j}^*
 \end{align*}

The next Lemma shows that some terms will vanish under the conditional expectation $E_{A_{j}}$.
\begin{lemma} The conditional expectation $E_{A_j}$ has the following properties: 
\begin{enumerate}
\item [(i)]$E_{A_{j}}(a_{k}^2(1-a^{2})(a_{k}^*)^2)=0$ for $k\neq j$
\item [(ii)]$E_{A_{j}}(a_{k}a_{l}(1-a^2)a_{l}^*a_{k}^*)=0$ unless $k=l=j$ 
\end{enumerate}
\end{lemma}
\begin{proof}
Let us start by $(i)$ and assume $k \neq j, $ 
\begin{align*}
E_{A_{j}}(au_{k}au_{k}(1-a^{2})u_{k}^*au_{k}^*a)& =aE_{A_{j}}(u_{k}au_{k}(1-a^{2})u_{k}^*au_{k}^*)a\\
&=a E_{A}(u_{k}au_{k}(1-a^{2})u_{k}^*au_{k}^*)a=0
\end{align*}
For $(ii)$ we will consider the following cases for $j,l,k$ :
\begin{enumerate}
\item If $j \neq k\neq l$ then $E_{A_{j}}(a_{k}a_{l}(1-a^2)a_{l}^*a_{k}^*)=0$
\item If $j=k$ then $E_{A_{j}}(a_{k}a_{l}(1-a^2)a_{l}^*a_{k}^*)=0$
\item If $j=l$ and $k\neq l$ $E_{A_{j}}(a_{k}a_{l}(1-a^2)a_{l}^*a_{k}^*)=0$
\end{enumerate}
In case (1), we deduce from above
$$E_{A_{j}}(a_{k}a_{l}(1-a^2)a_{l}^*a_{k}^*)= E_{A_{j}}(au_{k}au_{l}(1-a^2)u_{l}^*au_{k}^*a)
=aE_{A_{j}}(u_{k}au_{l}(1-a^2)u_{l}^*au_{k}^*)a   $$
Indeed  let $\omega \in A_j $ and $ \omega := E_A( \omega )  + \omega -E_A( \omega ) = E_A( \omega ) +   \mathring{w}$, then 
\begin{align*}
E_{A_{j}}&(u_{k}au_{l}(1-a^2)u_{l}^*au_{k}^*) = \tau ( u_{k}au_{l}(1-a^2)u_{l}^*au_{k}^* \omega ) \\
& = \underbrace{\tau (  u_{k}au_{l}(1-a^2)u_{l}^*au_{k}^* \mathring{\omega} ) }_{= 0}+\underbrace{ \tau ( u_{k}au_{l}(1-a^2)u_{l}^*au_{k}^* E_A(\omega ))}_{= 0}
\end{align*}
 by (2) and (4) in Remark \ref{rem33}.

In case (2) we have
\begin{align}
E_{A_{j}}(au_{j}au_{l}(1-a^{2})u_{l}^*au_{j}^*a)=au_{j}aE_{A_{j}}(u_{l}(1-a^{2})u_{l}^*)au_{j}^*a=0.
\end{align}
In case (3), when $j=l$ and $k\neq l$, we have 
\begin{align}
E_{A_{j}}(au_{k}au_{j}(1-a^{2})u_{j}^*au_{k}^*a)=aE_{A_{j}}(u_{k}au_{j}(1-a^{2})u_{j}^*au_{k}^*)a = 0
\end{align}
Since it's the same with case (1).
\end{proof}
Now we can prove the theorem by using the two lemma above.
\begin{proof}[Proof of Theorem \ref{counter}]
Assume $n \geq 2 $ and $\mathbb{E}_{wo} ( \mathcal{A} ) \leq \mathbb{E}_{wr}  ( \mathcal{A}) .$ Then by Lemma \ref{lem32}
$$\tau ( \mathbb{E}_{wo}( \mathcal{A})  )  = \tau( \mathbb{E}_{wr} (\mathcal{A} )) = \tau (a^2)^3 = 1 $$
implies $$  \tau ( \mathbb{E}_{wr} (\mathcal{A}  )- \mathbb{E}_{wo}( \mathcal{A})  )   = 0$$
Since $\tau$ is faithful, then 
$$x \geq 0, \tau(x) = 0 \Longrightarrow    x = 0.$$
We get $$\mathbb{E}_{wo} ( \mathcal{A}) =  \mathbb{E}_{wr}( \mathcal{A}).$$ 
Therefore, 
$$E_{A_{j}}( \mathbb{E}_{wo}- \mathbb{E}_{wr}  ) = 0. $$
Indeed 
\begin{align*}
E_{A_{j}}(\mathbb{E}_{wo}- \mathbb{E}_{wr} )&=\frac{-1}{n^{2}(n-1)}a_{j}^{2}(1-a^{2})(a_{j}^*)^2+\frac{1}{n(n-1)}a_{j}^2(1-a^{2})(a_{j}^{*})^2 \\
& = \frac{1}{n^2}a_{j}^2(1-a^{2})(a_{j}^{*})^2 .
\end{align*}
 Hence
$$ E_{A_{j}}(a_{j}^{2}(1-a^{2})(a_{j}^*)^2)=0$$
which means 
$$  a_{j}^{2}(1-a^{2})(a_{j}^*)^2=0 .$$
Therefore
$$ a^2=1.  $$
Then we get contradiction with the assumption.
\end{proof}

Following Theorem \ref{derivation},  we may replace the $\vertiii{\cdot}$ with operator norm for free products, we  have the following:
\begin{cor}
	Let $ \mathcal{A} =
	\{ a_i\}_{i=1}^n$ be freely independent operators such that $\tau(a_i^*a_i) = 1.$ If $\frac{1}{n}\|\sum a_i^* a_i - \tau(\sum a_i^* a_i  )  \| \leq \varepsilon  \text{ and }d \ll n,$
then 
\begin{enumerate}
	\item   $\|  \mathbb{E}_{wo,d}(\mathcal{A})- E( \mathbb{E}_{wo,d}(\mathcal{A})) \| \leq  O(d\varepsilon), $
	\item $\|  \mathbb{E}_{wo,d}(\mathcal{A}) \|\leq (1+  O(d\varepsilon)) \|\mathbb{E}_{wr,d}(\mathcal{A})\|$
\end{enumerate}
\end{cor}

\begin{rem}
	In comparison with Theorem \ref{counter}, a family of freely independent operator allows some tolerance between $\sum a_i^* a_i $ and $\tau(\sum a_i^* a_i $, then the symmetrized AGM still holds with a constant $1+  O(d\varepsilon)$. Moveover, when $n$ is closer and closer to infinity, $1+  O(d\varepsilon)$ gets close to 1. Then we obtain the symmetrized-AGM inequality. 
\end{rem}

\section{Incremental Gradient Method and the symmetric AGM inequality}
In this section we consider the Incremental Gradient Method (IGM) in higher dimensions where we demonstrate the error for sampling without replacement-method by using the upper bound of the symmetric AGM inequality. 
\subsection{Recursion formula for IGM} 

Let $x_*$ be a vector in $\mathbb{R}^m$ and set
$$y_i = a_i^* x_* + w_i, \text{for } i = 1, \cdots, n$$
where $a_i \in \mathbb{R}^m$ are test vectors and $w_i$ are i.i.d. Gaussian random variables with mean zero and variance $\rho^2.$

We compare sampling with-replacement versus without-replacement sampling for IGM on the cost function
\begin{equation}
	\min_x \sum_{i=1}^n (a_i^* x -y_i)^2.
\end{equation}
Suppose we walk over $k$ steps of IGM with constant step size $\gamma$ and we access the terms $i_1, \cdots, i_k$ in that order. Then we get
\begin{eqnarray*}
x_{i_{k}} &=& x_{i_{k-1}} - \gamma a_{i_{k}} (a_{i_{k}}^* x_{i_{k-1}} - y_{i_k} )\\
    &=& (I - \gamma a_{i_k} a_{i_k}^*) x_{i_{k-1}} + \gamma a_{i_k} y_{i_k}
\end{eqnarray*}
Subtracting $x_*$ from both sides of this equation, then gives

\begin{align*}
x_{i_k} - x_* &= (I - \gamma a_{i_k} a_{i_k}^*) x_{i_{k-1}} + \gamma a_{i_k} y_{i_k} -x_{*}\\
&=(I - \gamma a_{i_k} a_{i_k}^*) x_{i_{k-1}} +\gamma a_{i_k} (a_{i_{k}}^{*}x_{*}+w_{i_{k}})-x_{*}\\
&=(I- \gamma a_{i_k}a_{i_k}^*)(x_{i_{k-1}} - x_*) + \gamma a_{i_k} w_{i_k} \\
\end{align*}
Substitute by  $$x_{i_{k-1}}=(I - \gamma a_{i_{k-1}} a_{i_{k-1}}^*) x_{i_{k-2}} + \gamma a_{i_{k-1}} y_{i_{k-1}} \text{ and } y_{i_{k-1}}=a_{i_{k-1}}^{*}x_{*}+w_{i_{k-1}},$$then we have 
\begin{align*}
&x_{i_k} - x_* 
=(I- \gamma a_{i_k}a_{i_k}^*)\Big((I - \gamma a_{i_{k-1}} a_{i_{k-1}}^*) x_{i_{k-2}} + \gamma a_{i_{k-1}} (a_{i_{k-1}}^{*}x_{*}+w_{i_{k-1}}) - x_*\Big) + \gamma a_{i_k} w_{i_k}\\
&=(I- \gamma a_{i_k}a_{i_k}^*)\Big((I - \gamma a_{i_{k-1}} a_{i_{k-1}}^*) x_{i_{k-2}} - (I - \gamma a_{i_{k-1}} a_{i_{k-1}}^*) x_{*} +\gamma a_{i_{k-1}} w_{i_{k-1}} \Big)+  \gamma a_{i_{k}} w_{i_{k}}\\
&=(I- \gamma a_{i_k}a_{i_k}^*)(I - \gamma a_{i_{k-1}} a_{i_{k-1}}^*) (x_{i_{k-2}} -x_{*})+ (I- \gamma a_{i_k}a_{i_k}^*) \gamma a_{i_{k-1}} w_{i_{k-1}} + \gamma a_{i_{k}} w_{i_{k}}.
\end{align*}
By iteration over $k$, we obtain the following term 
\begin{align}\label{eq442}
x_{i_k} - x_* =\prod_{j=1}^k (I - \gamma a_{i_j} a_{i_j}^*) (x_0 - x_*) +  \sum_{l=1}^k \Big[\prod_{k \geq j > l} (I-\gamma a_{i_j}a_{i_j}^*)\gamma a_{i_l}\Big] w_{i_l}
\end{align}
\subsection{From incremental gradient method to symmetric-AGM}

Now we take the inner product $\langle x_k - x_*,x_k - x_*\rangle=\|x_k - x_*\|^{2}$ and take a partial expectation with respect to $w_{i}$ to estimate $\|x_k - x_*\|$ after $k$ steps mentioned in \eqref{eq442}. If $x, y$ are mutually independent random vectors such that $\mathbb{E}(y)=0$, then we have $ \mathbb{E}\langle x+y,x+y\rangle= \mathbb{E}(\langle x,x \rangle ) + \mathbb{E}\langle y,y\rangle$. Because of this property,   taking the expectation of \eqref{eq442},  we get
\begin{multline*}
 \mathbb{E}_{\text{wo}}[\|x_k - x_* \|^2   ] = \mathbb{E}_{\text{wo}}\Bigg[\|\prod_{j=1}^k   (I -  \gamma a_{i_j} a_{i_j}^*)(x_0 - x_*)\|^2\Bigg]  \\
 + \rho^2 \gamma^2 \sum_{l=1}^k  \mathbb{E}_{\text{wo}} \Bigg[ \| \prod_{k \geq j > l} (I-\gamma a_{i_j}a_{i_j}^*) a_{i_l}  \|^2  \Bigg]
\end{multline*}
Denote $A_{i_{j}} := I -  \gamma a_{i_j} a_{i_j}^*$, then $\{A_{i_{j}}\}$ be a family of a self adjoint operators and the above equation can be written as 
\begin{multline}\label{imp1}
 \mathbb{E}_{\text{wo}}[\|x_k - x_* \|^2   ] =\mathbb{E}_{\text{wo}}\Bigg[\|\prod_{j=1}^k   A_{i_{j}}(x_0 - x_*)\|^2\Bigg]  
 + \rho^2 \gamma^2 \sum_{l=1}^k \mathbb{E}_{\text{wo}} \Bigg[ \| \prod_{k \geq j > l} A_{i_{j}} a_{i_l}  \|^2  \Bigg] . 
\end{multline}
Expanding the square norm in the $\mathbb{E}_{\text{wo}},$ we get
\begin{align} \nonumber
 \mathbb{E}_{wo}[\|x_k - x_* \|^2  ] 
 & \leq \mathbb{E}_{wo} \Big [ \| A_k^*\cdots A_1^*A_1 \cdots A_k \|\Big ]  \| x_0 - x_*  \|^2 \\ \label{45}
 & + \rho^2 \gamma^2 \sum_{l=1}^k \mathbb{E}_{wo} \Big [ \| A_k^*\cdots A_{l+1}^*A_{l+1} \cdots A_k \| \Big ] \|a_{i_l}\|^2
\end{align}
Here we use $\mathcal{S}_d(\mathcal{A}) $ to replace the symbol $\mathbb{E}_{wo} \Big [ \| A_k^*\cdots A_1^*A_1 \cdots A_k \|\Big ]  $.  We have to split the two terms of \eqref{45} into  $l \leq k-1$ and $l = k$ :
\begin{align} \label{pp}
\mathbb{E}_{wo}[\|x_k - x_* \|^2  ] 
 & \leq \|\mathcal{S}_d(\mathcal{A})\|\cdot \|x_0 - x_*\|^2 \\ \nonumber
 &+\rho^2 \gamma^2 \Bigg(\underbrace{ \sum_{l=1}^{k-1} \mathbb{E}_{wo} \Big [ \| A_k^*\cdots A_{l+1}^*A_{l+1} \cdots A_k \| \Big ]}_{(I_l)} + 1\Bigg) \cdot  \sup_{j \in [n]} \|a_j\|^2
\end{align}
Note that the term $I_l$ misses out some indexes and hence is not exactly $\mathcal{S}_l$. We define $\mathcal{S}_{j,k}(\mathcal{A})  = \mathbb{E}_{wo} \Big [  A_k^*\cdots A_{j+1}^*A_{j+1} \cdots A_k  \Big]$ where $j < k$. Then \eqref{pp} can be reformulated as follows:
\begin{multline}\label{main1}
\mathbb{E}_{wo}[\|x_k - x_* \|^2  ] 
 \leq  
 \|\mathcal{S}_k(\mathcal{A})\|\cdot \|x_0 - x_*\|^2 
  +\rho^2 \gamma^2 \Bigg(  \sum_{l = 1}^{k-1}   \| \mathcal{S}_{l,k}(\mathcal{A})  \| + 1 \Bigg) \cdot  \sup_{j \in [n]} \|a_j\|^2
\end{multline}

Observe that 
\begin{enumerate}
	\item $\mathcal{S}_{l,k}= \text{avg} \Bigg( \sum\limits_{i_{1},...,i_{k-l}\in[n]-[l]}A_{i_{k-l}}^*...A_{i_{1}}^*A_{i_{1}}...A_{i_{k-l}}  \Bigg),$ and 
	\item $\mathcal{S}_{l,k} \leq   \sum A_{i_{k-l}}^*\mathcal{S}_{l +1,k} A_{i_{k-l}}.$
\end{enumerate}

\begin{lemma} 
	\begin{equation}
	\mathcal{S}_{l,k}  \leq C_{k,l}   \mathcal{S}_{k-l} , ~~C_{k,l} := \frac{(n...n-l+1)(n...n-(k-l)+1)}{n...n-k+1}.
	\end{equation}
\end{lemma}
 \begin{proof}
 	For fixed $l$, $ E_{wo} A_{i_{k}}^{*}...A_{i_{l+1}}^*A_{i_{l+1}}...A_{i_{k}} $ is  comparable to $\mathcal{S}_{k-l}.$ 
 	Thanks to the above observation, 
 	\begin{align*}
 	S_{l, k} = &\frac{1}{n...(n-k+1)}\sum\limits_{i_{k},...,i_{l+1}\in \{1,...,n\}-\{i_{1},...,i_{l}\}} A_{i_{k}}^{*}...A_{i_{l+1}}^*A_{i_{l+1}}...A_{i_{k}} \\
 	&\leq \frac{1}{n...(n-k+1)} \sum\limits_{i_{k},...,i_{l+1}\in \{1,...,n\}} A_{i_{k}}^{*}...A_{i_{l+1}}^*A_{i_{l+1}}...A_{i_{k}} \\
 	&\leq \frac{1}{n...(n-k+1)} \sum\limits_{i_{1},...,i_{l}} n...(n-(k-l)+1) \mathcal{S}_{k-l} \\
 	&=\underbrace{\frac{(n...n-l+1)(n...n-(k-l)+1)}{n...n-k+1}}_{C_{k,l}}\mathcal{S}_{k-l}.\qedhere
 	\end{align*}
 \end{proof}

 \begin{lemma} The constant $C_{k,l}$ havs the following property:
 	\begin{enumerate}
 		\item $C_{k,l}\simeq \exp(\frac{lk}{n-k}).$
 		\item $C_{k, k-i} = C_{k, i}$
 	\end{enumerate}
   
 \end{lemma}  
 \begin{proof} The (2) is trivial based on (1). So we just need to verify the part (1). Mutiply $\frac{1}{n^k}$ to the fraction form, we get 
 	$$C_{k,l}=\frac{{\frac{n...(n-l+1)}{n^{l}}}\frac{n...n-(k-l)+1}{n^{k-l}}}
 	{{\frac{n...n-k+1}{n^{k}}}}\cdot \frac{n^{l}~n^{k-l}}{n^{k}}\rightarrow .1$$
 	Let us discuss one of these terms
 	\begin{align*}
 	\alpha_{k} "=\ln \frac{n...n-k+1}{n^{k}}
 	=\sum\limits_{j=0}^{k-1}\ln \frac{n-j}{n}=\sum\limits_{j=0}^{k-1}\ln (1-\frac{j}{n}).
 	\end{align*}
 	Taking the logarithm of this constant $C_{k,l}$, we will have $\alpha_{l}+\alpha_{k-l}-\alpha_{k}$ such that 
 	\begin{align*}
 	\alpha_{l}+\alpha_{k-l}-\alpha_{k}&=\sum\limits_{j=0}^{l-1}\ln(1-\frac{j}{n})+\sum\limits_{j=0}^{k-l-1}\ln(1- \frac{j}{n})-\sum\limits_{j=0}^{k-1}\ln(1-\frac{j}{n})\\
 	&=\sum\limits_{j=0}^{k-l-1}\ln(1-\frac{j}{n})-\sum\limits_{j=l}^{k-1}\ln(1-\frac{j}{n})  = \sum\limits_{j=0}^{k-l-1} \ln \Big( 1+ \frac{l}{n- l - j}  \Big) \\
 	& \leq (k-l) \ln (1 + \frac{l}{n-k})  =  \ln  (1 + \frac{l}{n-k}) ^{k-l} 
 	\end{align*}
 	Therefore
 	\begin{equation} \label{eq41}
 	C_{k,l} = exp (\alpha_l + \alpha_{k-1} - \alpha_k) \leq   (1 + \frac{l}{n-k}) ^{k-l}  \leq \exp(\frac{l(k-l)}{n-k})  \qedhere
 	\end{equation}
 \end{proof}

\subsection{Convergent rate for the IGM}

\begin{theorem}\label{mainthm}
	For the IGM, $w_i$ are i.i.d. Gaussian random variables with mean zero and variance $\rho^2$, and the test vectors $a_i$ satify 
	\begin{enumerate}
		\item  $\frac{1}{n} \sum_{j=1}^{n} a_{i_j}a_{i_j}^* = \sigma I,$
		\item $ \sup \|a_{i_j}\|^2= \mu.$
	\end{enumerate}
	  For k iterations, we will have the following estimate
	\begin{multline} \label{main}
	\mathbb{E}_{wo}[\|x_k - x_* \|^2  ]  \leq \varphi^{k} \Big(1+ \frac{k(k-1)}{2n}(1+C_1) \Big)  \eta  \\
	+   \rho^2 \gamma^2 \mu \Big( \frac{1}{1- \varphi \exp( \frac{1}{n-k}  ) } +  C_2\varphi \exp(\frac{1}{n-k})  + 1 \Big)
	\end{multline}
	Here $\phi = 1 - 2\gamma \sigma + \gamma^2 \sigma \mu, \beta = \|x_0 - x_*\|, C_1, C_2$ are two constants related to $\gamma, \mu, \sigma.$
\end{theorem}

From now on,  we always assume  $\frac{1}{n} \sum_{j=1}^{n} a_{i_j}a_{i_j}^* = \sigma I, \sup \|a_{i_j}\|^2= \mu.$ Recall that we define $A_j $ as $ I -  \gamma a_{i_j} a_{i_j}^*$, and obtain the following relations

\begin{enumerate} 
\item $A_j^*A_j =  I - 2\gamma a_{i_j}a_{i_j}^* + \gamma^2 a_{i_j} a_{i_j}^* \|a_{i_j}\|^2  $
 \item $  \frac{1}{n} \sum_{j=1}^{n}  A_j^*A_j  = (1 - 2\gamma \sigma + \gamma^2 \sigma \mu) \cdot I := \varphi \cdot I $ \label{average1}
\item $\sup \|A_j\|^2  \leq \max\{|1- \gamma \mu|,1\}$
\end{enumerate}
By rescaling the operators  $A_j = \sqrt{\varphi} U_j$ in the average form \eqref{average1} and using the upper bound inequality \eqref{bdd1} for $\frac{1}{n}\sum U_j^*U_j = I$,  we can get 
\begin{equation}\label{eq42}
\|\mathcal{S}_k (\mathcal{A})\|  = \varphi^k \| \mathcal{S}_k (U)  \| \leq  \varphi^k (1+\underbrace{\frac{k(k-1)}{2n} (1+ C)}_{\vartriangle_{k}})= \varphi^k (1+\vartriangle_{k})
\end{equation}
here 
\begin{equation*}
C =  \frac{\sup \|A_j\|^2}{\varphi}, ~~ \vartriangle_k = \frac{k(k-1)}{2n} (1+ C).
\end{equation*}

Now,thanks to the inequality \eqref{main1}, we have
\begin{align} \nonumber
\mathbb{E}_{wo}[\|x_k - x_* \|^2  ] 
&\leq  
\|\mathcal{S}_k(\mathcal{A})\|\cdot \|x_0 - x_*\|^2 
+\rho^2 \gamma^2 \Bigg(  \sum_{l = 1}^{k-1}   \| \mathcal{S}_{l,k}(\mathcal{A})  \| + 1 \Bigg) \cdot  \sup_{j \in [n]} \|a_j\|^2 \\
&\leq \|\mathcal{S}_k(\mathcal{A})\|\cdot \|x_0 - x_*\|^2+  \rho^2 \gamma^2 \mu \sum_{l=1}^{k-1}   C_{k,l}\|\mathcal{S}_{k-l}(\mathcal{A})\|    + \rho^2 \gamma^2 \mu \\ \label{46}
& = \varphi^{k} \Big(1+ \vartriangle_k \Big)  \eta +   \rho^2 \gamma^2 \mu \sum_{i=1}^{k-1} C_{k, k-i} \|\mathcal{S}_i (\mathcal{A}) \| + \rho^2 \gamma^2 \mu  
\end{align}
The above equality\eqref{46} comes from the change of index by $k-i $ and the identity $C_{k, k-i} = C_{k, i}$. Next adapting \eqref{eq41} and \eqref{eq42}, we have 
\begin{align*}
\mathbb{E}_{wo}[\|x_k - x_* \|^2  ] & \leq \varphi^{k} \Big(1+ \vartriangle_k \Big)  \eta +   \rho^2 \gamma^2 \mu \sum_{i=1}^{k-1} \exp(\frac{i(k-i)}{n-k}  ) \| \mathcal{S}_i (\mathcal{A}) \| + \rho^2 \gamma^2 \mu  \\
&  \leq \varphi^{k} \Big(1+ \vartriangle_k \Big)  \eta +   \rho^2 \gamma^2 \mu  \sum_{i=1}^{k-1} \exp(\frac{i}{n-k}  ) \varphi^i(1+ \vartriangle_i)  + \rho^2 \gamma^2 \mu  
\end{align*}

\begin{lemma} The following holds with constant C related to n, k and $\varphi$, 
	\begin{equation}\label{47}
	\sum_{i=1}^{k-1} \exp(\frac{i}{n-k}  ) \varphi^i(1+ \vartriangle_i)  \leq  \frac{1}{1- \varphi \exp( \frac{1}{n-k}  ) } +  c \varphi \exp(\frac{1}{n-k})
	\end{equation}
\end{lemma}
\begin{proof} Denote $a: = \frac{1}{n-k} + \ln \varphi $,  the left side of \eqref{47} via changing the base of the power 
	\begin{align*}
	&\sum_{i=0}^{k-1} \exp(\frac{i}{n-k}  ) \varphi^i(1+ \vartriangle_i)  = \sum_{i=0}^{k-1}   \exp(\frac{i}{n-k}  ) \exp(i \ln \varphi ) (1+ \vartriangle_i) 
	= \sum_{i=0}^{k-1} \exp(ai)  (1+ \vartriangle_i)\\
	&= \sum_{i=0}^{k-1} \exp(ai)+ \sum_{i=0}^{k-1}  \exp(ai) \vartriangle_i  = \frac{1- \exp(ak)}{1-\exp(a)}  + \frac{1+C}{2n} \sum_{i=1}^{k-1}  \exp(ai) i(i-1)\\
	\end{align*} 
Denote $g(a) = \sum_{i=1}^{k-1}\exp(ai) i(i-1), F(x) = (x^2 - 2x +2) e^x , c_a  = \frac{a^2 - 2a +2}{a^3}$
\begin{equation*}
g(a)  \leq \int_1^{k} e^{ax} x^2 dx = \frac{-1}{a^3} (F(a) - F(ak)) \leq \frac{F(a)}{(-a)^3}\leq  c_a \varphi \exp(\frac{1}{n-k})
\end{equation*}
Therefore 
\begin{equation*}
\sum_{i=0}^{k-1} \exp(\frac{i}{n-k}  ) \varphi^i(1+ \vartriangle_i)  \leq  \frac{1}{1- \varphi \exp( \frac{1}{n-k}  ) } +  c \varphi \exp(\frac{1}{n-k})\qedhere
\end{equation*}
\end{proof}

\begin{proof}[Proof of Theorem \ref{mainthm}]
	Combined with the above inequality we have
	\begin{align*}
	&\mathbb{E}_{wo}[\|x_k - x_* \|^2  ]    \leq \varphi^{k} \Big(1+ \vartriangle_k \Big)  \eta +   \rho^2 \gamma^2 \mu \Big( \frac{1}{1- \varphi \exp( \frac{1}{n-k}  ) } +  c_a \varphi \exp(\frac{1}{n-k})   \Big)+ \rho^2 \gamma^2 \mu \\
	& \leq  \varphi^{k} \Big(1+ \frac{k(k-1)}{2n}(1+C) \Big)  \eta +   \rho^2 \gamma^2 \mu \Big( \frac{1}{1- \varphi \exp( \frac{1}{n-k}  ) } +  c_a \varphi \exp(\frac{1}{n-k})   \Big)+ \rho^2 \gamma^2 \mu
	\end{align*}
	Therefore the target esimate is under control of $\varphi, \rho \text{ and } \gamma, i.e. $  
	\begin{multline} \label{l}
	\mathbb{E}_{wo}[\|x_k - x_* \|^2  ]  \leq \varphi^{k} \Big(1+ \frac{k(k-1)}{2n}(1+C) \Big)  \eta \\
	+   \rho^2 \gamma^2 \mu \Big( \frac{1}{1- \varphi \exp( \frac{1}{n-k}  ) } +  c_a \varphi \exp(\frac{1}{n-k})  + 1 \Big)\qedhere
	\end{multline}
\end{proof}


\begin{rem} 
In order to make sure the inequality \eqref{l} hold, we need to guarantee that the radio $\varphi \in (0,1)$. By the definition of $\varphi$, we will have
\begin{equation}\label{domain}
 0 < 1 - 2 \gamma \sigma + \gamma^2 \sigma \mu  < 1 .
\end{equation}

	If $ \gamma < \frac{2}{\mu},$ we will have $\varphi < 1$ and $C = \frac{1}{\varphi}$ as well. Moreover, if $\sigma < \mu$, then $\gamma > 0$ exists.
\end{rem}
\begin{rem}\normalfont
	The suggested strategy to win an average convergent rate:
	Given $\delta >0$, then we can find a pair $(k_0, \eta_0 )$, s.t $k > k_0$ and $\eta < \eta_0 $, then
	$$\mathbb{E}_{wo}[\|x_k - x_* \|^2  ]    \leq \delta.$$ 
	Denote the terms in \eqref{l}
	\begin{multline*} 
	\mathbb{E}_{wo}[\|x_k - x_* \|^2  ]  \leq  \underbrace{\varphi^{k} \Big(1+ \frac{k(k-1)}{2n}(1+C) \Big)  \eta}_{(I)} \\
	+  \underbrace{ \rho^2 \gamma^2 \mu \Big( \frac{1}{1- \varphi \exp( \frac{1}{n-k}  ) } +  c_a \varphi \exp(\frac{1}{n-k})  + 1 \Big) }_{(II)}
	\end{multline*}
	\begin{itemize}
		\item The terms in the  (I) can be bounded by  $ (I)  \leq \varphi^k \beta(k)  < \delta $, if $k < n^{1/3}$.
		\item The terms in the bracket of (II) can be uniformly controlled by some constant C. Therefore  if we want to achieve $(II) \leq C\rho^2 \gamma^2 \mu  < \frac{\delta}{2},$ we just need to require $\gamma$ small enough.
	\end{itemize}
 We can  run the regular IGM algorithm with replacement sampling of the regular test vectors $a_i$ as long as $ k< n^{1/3}$.If this is not the case,  then we enlarge the test vectors $a_{li} = a_i, l = 1, \cdots m.$ So the total number of the sampling pool becomes $nm$ where $m \geq \frac{k^3}{n}$. Therefore, we can run the IGM with new test vectors $a_{l,i}$.
\end{rem}

 \subsection{Examples satifying the assumptions of the main result}

\begin{exam}\normalfont
Let G denote a compact group and  $\pi: G \rightarrow U_d$ be an affine representation that maps an element in  G to the orthogonal group $U_d$, s.t. $\pi(g)\pi(h) = e^{i \phi (g,h)}\pi(gh).$ An affine representation is isotropic if $\int \pi(g)x\pi(g)^* dg = \frac{tr(x)}{d}1_d$.

Let $|G| = n$. Then $\exists h \in \mathbb{C}^d, \|h\| = \sqrt{d}$
\begin{equation}\label{isot}
	\frac{1}{|G|} \sum_g (g h\rangle \langle hg^{-1}) = \frac{tr (h\rangle \langle h)}{d} Id
\end{equation}
\begin{enumerate}
	\item Let $a_g = gh \rangle \langle e_0$, here $e_0$ is an unital element in $\mathbb{C}^d.$ Observe that $\sigma = 1, \mu = d,\varphi = 1 - 2 \gamma + \gamma^2 d, C= \frac{1}{\varphi}$ Then \eqref{main} will hold
	in the domain of $\gamma \in (0 , \frac{2}{d})$ in \eqref{domain}.
	\item Let $a_g = gh \rangle \langle  g h^{-1}$. Then $\sigma = d, \mu = d^2, \varphi = 1- 2\gamma d + \gamma d^3, C = \frac{1}{\varphi}.$ Then \eqref{main} will hold
	in the domain of $\gamma \in (0 , \frac{2}{d^2})$ in \eqref{domain}.
\end{enumerate}
\end{exam}
\vspace{0.2in}
Spherical designs were defined by Delsarte-Goethals-Seidel in 1977 \cite{delsarte1977spherical}.
We consider a finite subset $\mathbb{X}$ on the unit sphere $\mathbb{S}^{n-1}$ in n-dimensional Euclidean space $\mathbb{R}^n$. Let t be a natural number. A finite subset $ \mathbb{X} \subset  \mathbb{S}^{n-1} $ is called a spherical t-design if 
\begin{equation} \label{design}
\frac{1}{|\mathbb{S}^{n-1}|} \int_{\mathbf{x} \in\mathbb{S}^{n-1} } f( \mathbf{x} ) d \sigma(\mathbf{x}) = \frac{1}{| \mathbb{X} |}  \sum_{\mathbf{u} \in \mathbb{X} } f(\mathbf{u})
\end{equation}
holds for any polynomial $f( \mathbf{x} ) = f( x_1, x_2, \cdots, x_n )$ of degree at most t, with the usual integral on the unit sphere.
\begin{exam}\normalfont
For 2-designs we obtain 
\begin{equation}
\int_{\mathbb{S}^{d-1}} p(x) d\sigma(x)  = \frac{Id}{d}
\end{equation}
	Therefore taking $p(x ) = | x \rangle \langle x|$, by \eqref{design}  we have 
	\begin{equation*}
	 \frac{1}{M} \sum_{i=1}^M x_i \rangle \langle x_i= \frac{Id}{d}
	\end{equation*}
	
	Let $x_i = \sqrt{w_i} |\phi_i \rangle \langle \phi_i| $ with unitary $\phi_i$. Then $\exists h  \in \mathbb{S}^d, \sup |\langle h, \phi_i\rangle |^2\leq 1 - \epsilon$
	\begin{equation*}
	\frac{\|h\|_2}{d} =\frac{1}{M} \sum_i w_i |\langle h, \phi_i \rangle|^2 \leq \sup w_i \cdot (1-\epsilon) <  \sup w_i
	\end{equation*}	
	$\sup \|x_i\| = \sup w_i $. Then we know $\sigma = \frac{1}{d}, \mu = \sup w_i, \varphi = 1- 2\frac{\gamma}{d} + \frac{\gamma^2 \sup w_i}{d} , C = \frac{1}{\varphi}.$ Then \eqref{main} will hold
	in the domain of $\gamma \in (0 , \frac{2}{\sup w_i})$ in \eqref{domain}.
\end{exam}
\mbox{}
\nocite{*}
\bibliographystyle{plain}
\bibliography{agm2}

\begin{thebibliography}{10}

\bibitem{2017arXiv170300546A}
W.~{Albar}, M.~{Junge}, and M.~{Zhao}.
\newblock {Noncommutative versions of the arithmetic-geometric mean
  inequality}.
\newblock {\em ArXiv e-prints}, March 2017.

\bibitem{andrews1998theory}
George~E Andrews.
\newblock {\em The theory of partitions}, volume~2.
\newblock Cambridge university press, 1998.

\bibitem{delsarte1977spherical}
Philippe Delsarte, Jean-Marie Goethals, and Johan~Jacob Seidel.
\newblock Spherical codes and designs.
\newblock {\em Geometriae Dedicata}, 6(3):363--388, 1977.

\bibitem{emmrich1999discrete}
Etienne Emmrich.
\newblock Discrete versions of gronwall's lemma and their application to the
  numerical analysis of parabolic problems.
\newblock 1999.

\bibitem{hardy1952inequalities}
Godfrey~Harold Hardy, John~Edensor Littlewood, and George P{\'o}lya.
\newblock {\em Inequalities}.
\newblock Cambridge university press, 1952.

\bibitem{israel2016arithmetic}
Arie Israel, Felix Krahmer, and Rachel Ward.
\newblock An arithmetic--geometric mean inequality for products of three
  matrices.
\newblock {\em Linear Algebra and its Applications}, 488:1--12, 2016.

\bibitem{junge2006operator}
Marius Junge.
\newblock Operator spaces and araki-woods factors: a quantum probabilistic
  approach.
\newblock {\em International Mathematics Research Papers}, 2006, 2006.

\bibitem{pecaric2005generalization}
J~Pecaric, Jiajin Wen, Wan-lan Wang, and T~Lu.
\newblock A generalization of maclaurin's inequalities and its applications.
\newblock {\em Mathematical Inequalities and Applications}, 8(4):583, 2005.

\bibitem{pisier2000inequality}
Gilles Pisier et~al.
\newblock An inequality for $p$-orthogonal sums in non-commutative $ l\_
  $\{$p$\}$ $.
\newblock {\em Illinois Journal of Mathematics}, 44(4):901--923, 2000.

\bibitem{recht2012beneath}
Benjamin Recht and Christopher R{\'e}.
\newblock Beneath the valley of the noncommutative arithmetic-geometric mean
  inequality: conjectures, case-studies, and consequences.
\newblock {\em arXiv preprint arXiv:1202.4184}, 2012.

\bibitem{rota1964foundations}
Gian-Carlo Rota.
\newblock On the foundations of combinatorial theory i. theory of m{\"o}bius
  functions.
\newblock {\em Probability theory and related fields}, 2(4):340--368, 1964.

\bibitem{voiculescu2000lectures}
Dan Voiculescu.
\newblock Lectures on free probability.
\newblock {\em Lectures Notes in Mathematics}, 1738, 2000.

\end{thebibliography}
\end{document}